\documentclass{svproc}
\usepackage{amsmath,amssymb, eucal, xcolor, tikz-cd}
\usepackage{enumerate, url, hyperref}
\newcommand\R{\ensuremath{\mathbb{R}}}
\newcommand\Z{\ensuremath{\mathbb{Z}}}
\newcommand\esf{\ensuremath{\mathbb{S}}}
\newcommand\g{\ensuremath{\mathtt{g}}}
\newcommand\tg{\ensuremath{\widetilde{\mathtt{g}}}}
\newcommand\hg{\ensuremath{\widehat{\mathtt{g}}}}
\renewcommand\ker{\operatorname{Ker}}

\usepackage{url}

\begin{document}
\mainmatter              
\title{Compact plane waves \\ with parallel Weyl curvature}
\titlerunning{Compact plane waves with parallel Weyl curvature}  
%
\author{Ivo Terek}
\authorrunning{Ivo Terek} 
%
\tocauthor{Ivo Terek}

\institute{The Ohio State University, Columbus OH 43210, USA  \\[1em] Current address: Williams College, Williamstown MA 01267, USA\\
\email{it3@williams.edu}.}

\maketitle              

\begin{abstract}
This is an exposition of recent results --- obtained in joint work with Andrzej Derdzinski --- on \emph{essentially conformally symmetric} (ECS) manifolds, that is, those pseudo-Riemannian manifolds with parallel Weyl curvature which are not locally symmetric or conformally flat. In the 1970s, Roter proved that while Riemannian ECS manifolds do not exist, pseudo-Riemannian ones do exist in all dimensions $n\geq 4$, and realize all indefinite metric signatures. The local structure of ECS manifolds is known, and every ECS manifold carries a distinguished null parallel distribution $\mathcal{D}$, whose rank is always equal to $1$ or $2$. We review basic facts about ECS manifolds, briefly discuss the construction of compact examples, and outline the proof of a topological structure result: outside of the locally homogeneous case and up to a double covering, every compact rank-one ECS manifold is a bundle over $\esf^1$ whose fibers are the leaves of $\mathcal{D}^\perp$. Finally, we mention some classification results for compact rank-one ECS manifolds.
\end{abstract}

\section{Introduction and History}

By an \emph{essentially conformally symmetric} (ECS) manifold \cite{DR_padge}, we mean a pseudo-Riemannian manifold $(M,\g)$ of dimension $n\geq 4$ with parallel Weyl curvature ($\nabla W = 0$), which is not locally symmetric ($\nabla R\neq 0$) or conformally flat ($W\neq 0$). The reason why the latter conditions are excluded from the definition of an ECS manifold is because both $\nabla R=0$ and $W=0$ trivially imply that $\nabla W=0$. As Roter showed in 1977, this is not really something artificial to do: a \emph{Riemannian} manifold with parallel Weyl curvature must necessarily be locally symmetric or conformally flat \cite[Theorem 2]{DR_tensor}. Roter also showed that ECS manifolds do exist in all dimensions $n\geq 4$, and realizing all possible indefinite metric signatures \hbox{\cite[Corollary 3]{R_colloq}}.

The condition $\nabla W = 0$, being one of the natural differential conditions to be imposed in the irreducible components of the curvature tensor (in the sense of Besse, cf. \cite[Chapter 16]{Besse}), has been investigated and used by a number of authors --- these include Cahen and Kerbrat \cite{Cahen-Kerbrat}, Mantica and Suh \cite{Mantica-Suh}, Suh et al. \cite{Suh}, Hotlo\'{s} \cite{Hotlos}, Deszcz et al. \cite{Deszcz2,Deszcz3}, among others. Some techniques used in the study of ECS manifolds also have been applied to obtain results for more general classes of manifolds, as shown, for instance, in \cite{AG_jgp}, \cite{CGSV}, \cite{CGVV}, \cite{Deszcz1}, and \cite{T_pams}.

ECS manifolds are naturally sorted into two types. Namely, we define the \emph{rank} of $(M,\g)$ \cite{DT_pems} to be the rank of its \emph{Olszak distribution} $\mathcal{D}$, which is given by
\begin{equation}\label{eqn:Olszak}
  \mathcal{D}_x = \{v\in T_xM \mid \g_x(v,\cdot) \wedge W_x(v',v'',\cdot,\cdot) = 0,\mbox{ for all }v',v''\in T_xM\}
\end{equation}for all $x\in M$. The distribution $\mathcal{D}$ was originally introduced by Olszak \cite{O_zesz} for the study of conformally recurrent manifolds but, in the ECS case, it turns out that $\mathcal{D}$ is a null parallel distribution whose rank is always equal to $1$ or $2$. For this reason, one speaks of \emph{rank-one/rank-two ECS manifolds}.

Along with $\mathcal{D}$, the quotient vector bundle $\mathcal{D}^\perp/\,\mathcal{D}$ over $M$ also plays a central role in the study of ECS manifolds. Namely,
\begin{equation}\label{eqn:flat_D_Dquot}
  \parbox{.76\textwidth}{the Levi-Civita connection of $\,(M,\g)\,$ induces connections in both $\,\mathcal{D}\,$ and $\,\mathcal{D}^\perp/\,\mathcal{D}\,$, with the latter connection being always flat, and the former one being flat when $(M,\g)$ is of rank one,}
\end{equation}
cf. \cite[Lemma 2.2]{DR_bbms}. There, it is also shown that the Ricci endomorphism of $(M,\g)$ is $\mathcal{D}$-valued, regardless of the rank of $(M,\g)$.

Every Lorentzian ECS manifold $(M,\g)$, unable to carry a null distribution of rank two, must be of rank one. Flatness of the connections induced in $\mathcal{D}$ and $\mathcal{D}^\perp/\,\mathcal{D}$ then means that, up to a double isometric covering (taken to globally trivialize $\mathcal{D}$ if needed), $(M,\g)$ must be a \emph{pp-wave spacetime}. Such spacetimes, introduced by Ehlers and Kundt in the 1960s \cite{EK}, remain an active topic of research: see \cite{CFH,Flores-Sanchez}, \cite{AR_lmp,Lejmi}, \cite{Araneda,Roche}, and also the recent preprints \cite{HIMZ_arxiv,HMZ_arxiv}.  While solving the partial differential equation $\nabla W =0$ in Brinkmann coordinates does lead to a local classification result for Lorentzian ECS manifolds, the local structure of ECS manifolds of either rank has been determined in full generality by Derdzinski and Roter \cite{DR_bbms}. 

In view of the above, our attention will be mostly focused on global features of rank-one ECS manifolds; little is known about the rank-two case beyond their local structure and Theorem \ref{teo:exist_dil_examples} in Section \ref{sec:cpct_ex}. One very natural question is whether compact manifolds can admit ECS metrics or not. The first examples of compact ECS manifolds were provided by Derdzinski and Roter in 2010 \cite[Theorem 1.1]{DR_agag}, in all dimensions of the form $n=3k+2$ with $k\geq 1$, and they realize all indefinite metric signatures. Such examples, arising as isometric quotients of what we call here \emph{standard ECS plane waves} (see Section \ref{sec:models}), also present distinct geometric and topological features: they are all of rank one, geodesically complete, not locally homogeneous, and diffeomorphic to total spaces of torus bundles over $\esf^1$. The completeness conclusion in Lorentzian signature is not surprising: as shown by Leistner and Schliebner, compact pp-wave spacetimes are always complete \cite{LS_mann}; more generally, Mehidi and Zeghib have shown that a compact Lorentzian manifold carrying a null parallel vector field must be complete \cite{MZ_arxiv}.

\pagebreak

Some further questions arise:
\begin{enumerate}[(i)]
\item Are there compact ECS manifolds of dimensions $n\geq 5$ other than those of the form $n=3k+2$ with $k\geq 1$?
\item Are there compact four-dimensional ECS manifolds?
\item Must compact ECS manifolds be geodesically complete?
\item Can a compact ECS manifold be locally homogeneous?
\item Must compact rank-one ECS manifolds be bundles over $\esf^1$?
\end{enumerate}

Our recent progress in the topic consists in providing answers and partial answers to the above questions. In more detail:
\begin{enumerate}[(i)]
\item Yes. There are compact rank-one ECS manifolds of all dimensions $n\geq 5$, realizing all indefinite metric signatures \cite[Theorem A]{DT_mofm}.
\item There are no four-dimensional compact rank-one ECS manifolds \cite[Corollary F]{DT_agag}. However, it is still an open question whether there are four-di\-men\-si\-o\-nal compact rank-two ECS manifolds.
\item No. There are geodesically incomplete compact rank-two ECS manifolds of all \emph{odd} dimensions $n\geq 5$, with semi-neutral metric signature \cite[Theorem 6.1]{DT_advg}. However, in a \emph{generic} sense, compact rank-one ECS manifolds are complete \cite[Theorem E]{DT_agag}. (We elaborate on the meaning of `genericity' in Section \ref{sec:class_res}.)
\item Yes. The examples from \cite[Theorem 6.1]{DT_advg} mentioned in (iii) happen to be all locally homogeneous, although we also present incomplete non locally-homogeneous ones in \cite[Theorem B.1]{DT_advg}.
\item Outside of the locally homogeneous case and up to a double isometric covering, yes \cite[Theorem A]{DT_pems}. In addition, the fibers of the resulting bundle projection $M \to \esf^1$ are the leaves of $\mathcal{D}^\perp$. It is an open question whether these conclusions can be extended to the locally homogeneous case as well without the additional assumption that $\mathcal{D}^\perp$ has a compact leaf.
\end{enumerate}

\medskip

\noindent {\bf Outline:} In Section \ref{subsec:td-dich}, we review the \emph{translational-dilational} terminology \cite{DT_pm}, used when referring to the two distinct classes of compact ECS manifolds mentioned in the answers to (i) and (iii)-(iv) above.

In both \cite{DT_mofm} and \cite{DT_advg}, the resulting compact ECS manifolds again arise as isometric quotients of standard ECS plane waves, which we review in Section \ref{sec:models}. In particular, we also describe the isometry group of such manifolds.

Next, in Section \ref{sec:cpct_ex}, we introduce a family $\{{\rm G}(\sigma)\}_{\sigma \in {\rm S}}$ of subgroups of the isometry group of a standard ECS plane wave $(\widehat{M},\hg)$, and present in Theorem \ref{teo:criterion_compact} a criterion for the existence of subgroups $\Gamma$ of a given ${\rm G}(\sigma)$ for which the quotient $\widehat{M}/\Gamma$ is smooth and compact. Using it, we then outline the constructions given in \cite{DT_mofm} and \cite{DT_advg}. 

Section \ref{eqn:bundle_structure} is devoted to discussing the proof of the topological structure result mentioned in (v) above \cite{DT_pems}, and the concepts used in it; special emphasis is given to the dichotomy property \cite{D_rjm} for a codimension-one foliation $\mathcal{V}$ in a smooth manifold $M$, which is established and applied for $\mathcal{V} = \mathcal{D}^\perp$.

Finally, we conclude the presentation with Section \ref{sec:class_res}, elaborating on the notion of \emph{genericity} mentioned in the answer to (iii), and stating some classification results for generic compact rank-one ECS manifolds \cite{DT_pm,DT_agag}. Such results culminate in the answer to (ii).

\medskip

\noindent {\bf Acknowledgements:} It is a pleasure to thank Andrzej Derdzinski and Paolo Piccione for their very helpful comments on earlier drafts of this text.

\section{The translational-dilational holonomy dichotomy}\label{subsec:td-dich} The new compact ECS examples presented in \cite[Theorem A]{DT_mofm} and \cite[Theorems 6.1 and B.1]{DT_advg}, discussed in Section \ref{sec:cpct_ex}, are sorted out into two classes. Here, we discuss them in more generality: a \emph{weakly pp-wave manifold} $(M,\g)$, that is, a pseudo-Riemannian manifold carrying a null parallel distribution $\mathcal{P}$ of rank one satisfying \eqref{eqn:flat_D_Dquot} (for $\mathcal{D}$ replaced with $\mathcal{P}$), will be called \emph{translational} or \emph{dilational}, according to whether the ho\-lo\-no\-my group of the flat connection induced in $\mathcal{P}$ is finite or infinite. Every \emph{Ricci-recurrent} ECS manifold --- that is, those for with $\nabla {\rm Ric} = \theta\otimes {\rm Ric}$ for some $1$-form $\theta$, outside of the zero-set of ${\rm Ric}$ --- is weakly pp-wave, with $\mathcal{P}$ being the Olszak distribution $\mathcal{D}$ or the image of the Ricci endomorphism, according to whether its rank is one or two; this is a direct consequence of Roter's local structure theorem \cite{R_colloq}, and we describe the local model in Section \ref{sec:models}.

To understand the reason for the names `translational' and `dilational' more concretely, let $\pi\colon \widetilde{M} \to M$ be the universal covering of $M$ and set $\tg = \pi^*\g$. Then $(\widetilde{M},\tg)$ is also weakly pp-wave, with a distribution $\widetilde{\mathcal{P}}$ projecting onto $\mathcal{P}$ and also satisfying \eqref{eqn:flat_D_Dquot}. As $\widetilde{M}$ is simply connected, we may fix \vspace{-.5em}
\begin{equation}\label{t}\vspace{-.5em}
  \parbox{.75\textwidth}{a smooth function $t\colon \widetilde{M}\to I$, surjective onto some open interval $I\subseteq \R$, whose gradient is null, parallel, and spans $\widetilde{\mathcal{P}}$.}
\end{equation}
Writing $M = \widetilde{M}/\Gamma$ for some group $\Gamma \cong \pi_1(M)$ acting on $(\widetilde{M},\tg)$ freely and properly discontinuously by deck isometries, we see that for every $\gamma \in \Gamma$ there is $(q,p) \in {\rm Aff}(\R)$ such that $t\circ \gamma = qt+p$. This gives rise to two group homomorphisms
\begin{equation}\label{eqn:two_homs}
{\rm i})~ \Gamma \ni \gamma \mapsto (q,p)\in {\rm Aff}(\R)\quad\mbox{and}\quad {\rm ii}) ~\Gamma \ni \gamma \mapsto q\in \R\smallsetminus \{0\}.
\end{equation}
As the image of (\ref{eqn:two_homs}-ii) equals the holonomy group of the flat connection induced in $\mathcal{P}$, cf. \cite[Lemma 3.1]{DT_pm} (whose proof does not use compactness of $M$ or the parallel Weyl curvature condition), the two possibilities for $(M,\g)$ read: \vspace{-1em}
\begin{equation}\label{eqn:dich-q}
\vspace{-1em}  \parbox{.65\textwidth}{
    \begin{enumerate}[(i)]
    \item translational: $|q|=1$ for every \hbox{$\gamma \in \Gamma$}, and
    \item dilational: $\,|q|\neq 1\,$ for some $\gamma \in \Gamma$.
    \end{enumerate}
}\end{equation}
On the other hand, the image of (\ref{eqn:two_homs}-i) consists (up to an index-two subgroup) only of translations, or of dilations from a same fixed point, according to whether $(M,\g)$ is translational or dilational.

There are two last relevant consequences of \eqref{eqn:dich-q} for ECS manifolds.

\begin{proposition}
  Let $(M,\g)$ be a Ricci-recurrent ECS manifold. Then:
  \begin{enumerate}[\normalfont(i)]
  \item if $(M,\g)$ is Lorentzian, it is translational.
  \item if $(M,\g)$ is translational and compact, it cannot be locally homogeneous.
  \end{enumerate}
\end{proposition}

\begin{proof}
  Consider first (i), with $\mathcal{P} = \mathcal{D}$, as well as the vector space $V$ of parallel sections of $\widetilde{\mathcal{D}}^\perp/\widetilde{\mathcal{D}}$. Note that $\dim V = n-2$, and that $\tg$ induces a pseudo-Euclidean inner product $\langle\cdot,\cdot\rangle$ on $V$. If $Z$ is any vector field on $\widetilde{M}$ such that ${\rm d}t(Z)=1$, and $W$ is the Weyl tensor of $(\widetilde{M},\tg)$, the operator
  \begin{equation}\label{eqn:A}
\parbox{.65\textwidth}{$A\colon V\to V$, given by $A(X+\widetilde{\mathcal{D}}) = W(Z,X)Z+ \widetilde{\mathcal{D}},$}
\end{equation}is well-defined, traceless, self-adjoint, and independent of the choice of $Z$; see \cite[Section 4]{DR_jgp} and \cite[Section 5]{DT_agag}. The derivative of any element $\gamma\in \Gamma$, regarded as a deck isometry of $(\widetilde{M},\tg)$, induces an isometry $C$ of $(V,\langle\cdot,\cdot\rangle)$ with $CAC^{-1} = q^2A$, where $q$ is the (\ref{eqn:two_homs}-ii)-image of $\gamma$. When $(M,\g)$ is dilational, $A$ must necessarily be nilpotent, being conjugate to a nontrivial multiple of itself, cf. (\ref{eqn:dich-q}-ii). On the other hand, $(V,\langle\cdot,\cdot\rangle)$ is Euclidean when $(M,\g)$ is Lorentzian, and in this case the only self-adjoint nilpotent operator is $A=0$. As $W=0$ whenever $A=0$, (i) must hold.

To address (ii), we first observe that
\begin{equation}\label{eqn:Ric_f}
  \parbox{.66\textwidth}{${\rm Ric} = (2-n)f(t)\,{\rm d}t\otimes {\rm d}t$ in $\widetilde{M}$, for a suitable smooth function $f\colon \widetilde{M} \to \R$, which is locally a function of $t$.}
\end{equation}
Here, $f(t)$ stands for the composition $f\circ t$. Indeed, the Ricci endomorphism of $(\widetilde{M},\tg)$ is $\widetilde{\mathcal{D}}$-valued and self-adjoint, while the relation ${\rm d}^\nabla{\rm Ric} = 0$ (implied by $\nabla W = 0$) trivially leads to ${\rm d}f\wedge {\rm d}t=0$. In addition, we have
\begin{equation}\label{eqn:Ric-inv}
  \parbox{.88\textwidth}{$q^2f(qt+p) = f(t)$ for all $\gamma \in \Gamma$, with $(q,p)$ being the (\ref{eqn:two_homs}-ii)-image of $\gamma$,}
\end{equation}due to $\Gamma$-invariance of ${\rm Ric}$. Thus, when $(M,\g)$ is locally homogeneous, $|f|^{1/2}{\rm d}t$ is a closed and $\Gamma$-invariant $1$-form without zeros. The existence of such a $1$-form, in turn, implies that the level sets of $t$ are connected and coincide with the leaves of $\widetilde{\mathcal{D}}^\perp$, making $f$ a global function of $t$ \cite[Lemma 7.2]{DT_pems}. As $f$ is nowhere-vanishing and $(|f|^{-1/2})^{\boldsymbol\cdot\boldsymbol\cdot}=0$ by \cite[Theorem 7.3]{DT_pems}, we may assume that $f(t) = f(1)/t^2$ after replacing $t$ with an affine function of $t$ if necessary. On the other hand, if at the same time $(M,\g)$ is translational, $f$ is $\Gamma$-invariant by \eqref{eqn:Ric-inv} and (\ref{eqn:dich-q}-i), and thus survives on the compact quotient $M = \widetilde{M}/\Gamma$ as an unbounded and continuous function. This contradiction proves (ii).
\end{proof}

\section{Standard ECS plane waves}\label{sec:models}

\subsection{Setup and terminology}

In this section, following \cite{R_colloq}, we fix the following data: an integer $n\geq 4$, a pseudo-Euclidean vector space $(V,\langle\cdot,\cdot\rangle)$ of dimension $n-2$, a nonzero traceless and self-adjoint operator $A\colon V\to V$, an open interval $I\subseteq \R$, and a nonconstant smooth function $f\colon I\to \R$. With this in place, we consider the simply connected $n$-dimensional pseudo-Riemannian manifold
\begin{equation}\label{eqn:model}
  (\widehat{M}, \hg) = (I\times \R \times V,\, \kappa\,{\rm d}t^2 + {\rm d}t\,{\rm d}s + \langle\cdot,\cdot\rangle),
\end{equation}where $\kappa\colon I\times \R\times V\to \R$ is given by $\kappa(t,s,v) = f(t)\langle v,v\rangle + \langle Av,v\rangle$. Here, ${\rm d}t$, ${\rm d}s$, and $\langle\cdot,\cdot\rangle$ are identified with their pullbacks to $\widehat{M}$. Note that the null coordinate vector field $\partial_s$ is parallel, due to its being obviously a Killing vector field and corresponding under $\hg$ to the closed $1$-form ${\rm d}t/2$. We will repeatedly refer to
\begin{equation}\label{eqn:distr_P}
  \parbox{.77\textwidth}{the null parallel rank-one distribution $\mathcal{P}$ on $\widehat{M}$ spanned by $\partial_s$.}
\end{equation}
A routine computation also shows that
\begin{equation}\label{eqn:completeness_model}
  \parbox{.63\textwidth}{$(\widehat{M},\hg)$ is geodesically complete if and only if $I=\R$.}
\end{equation}
In addition, it is well-known that $(\widehat{M},\hg)$ is an ECS manifold, and that every point in a rank-one ECS manifold has a neighborhood isometric to an open subset of a suitable $(\widehat{M},\hg)$, cf. \cite[Theorem 4.1]{DR_bbms}. As $(t,s,v)$ are global Brinkmann coordinates for $(\widehat{M},\hg)$ and the function $\kappa$ is quadratic in the variable $v$, we will call $(\widehat{M},\hg)$ a \emph{standard ECS plane wave}. We emphasize that our standard ECS plane waves, which obviously include the Lorentzian ones, may have arbitrary indefinite metric signature.

In \cite{DT_pm} and \cite{DT_agag}, such $(\widehat{M},\hg)$ were called `rank-one ECS models.' To justify the change in terminology here, we observe that $(\widehat{M},\hg)$ does not necessarily have rank one as claimed in \cite{DR_bbms}, but instead
\begin{equation}\label{eqn:rank-correction}
  \parbox{.63\textwidth}{the rank of $(\widehat{M},\hg)$ equals $1$ or $2$ according to whether ${\rm rank}(A) >1$ or ${\rm rank}(A)=1$, respectively,}
\end{equation}
cf. \cite{DT_corr}. Indeed, the only possibly-nonzero components of the Weyl operator of $(\widehat{M},\hg)$ acting on bivectors are given by $W(\partial_t \wedge \partial_j) = 2\partial_s \wedge A\partial_j$, for all indices $j=1,\ldots, n-2$, where $(x^1,\ldots, x^{n-2})$ are arbitrary linear coordinates in $V$ \cite[p. 93]{R_colloq}. We may now compute the Olszak distribution $\mathcal{D}$ of $(\widehat{M},\hg)$, using \eqref{eqn:Olszak}. If $u=a\partial_t+b\partial_s + u_0$ lies in $\mathcal{D}$, with $a,b\in \R$ and $u_0 \in V$, the conditions $u\wedge \partial_s\wedge A\partial_j=0$ directly imply that $a=0$ and $u_0\wedge \partial_s\wedge w = 0$ for all $w\in {\rm Im}(A)$, so that $u_0 = 0$ whenever ${\rm rank}(A)>1$, while $b$ remains free; when ${\rm rank}(A) = 1$, we instead obtain that $u_0 \in {\rm Im}(A)$. Thus, for $\mathcal{P}$ given in \eqref{eqn:distr_P}, we have $\mathcal{D} = \mathcal{P}$ when ${\rm rank}(A)>1$, and $\mathcal{D} = \mathcal{P}\oplus {\rm Im}(A)$ when ${\rm rank}(A)=1$, finally proving \eqref{eqn:rank-correction}.

The standard ECS plane waves of the form \eqref{eqn:model} with ${\rm rank}(A)=1$ form a narrow --- but relevant --- class of rank-two ECS manifolds; see \cite{DR_tohoku} for more details on the local structure of rank-two ECS manifolds.

\subsection{The full isometry group}\label{sec:full_iso}

In order to describe the isometry group of an $n$-dimensional standard ECS plane wave $(\widehat{M},\hg)$ defined as in \eqref{eqn:model}, two ingredients are needed.

The first one is the subgroup $\mathrm{S}$ of the direct product ${\rm Aff}(\R)\times {\rm O}(V,\langle\cdot,\cdot\rangle)$, consisting of all triples $\sigma = (q,p,C)$ having $CAC^{-1}=q^2A$, $qt+p\in I$ and $f(t)=q^2f(qt+p)$ for all $t\in I$. The second one is
\begin{equation}\label{eqn:symp}
  \parbox{.86\textwidth}{the $2(n-2)$-dimensional symplectic vector space $(\mathcal{E},\Omega)$ of solutions $u\colon I\to V$ of the ordinary differential equation $\ddot{u}(t) = f(t)u(t) + Au(t)$.}
\end{equation}
Here, $\Omega$ is defined by $\Omega(u,w) = \langle \dot{u},w\rangle - \langle u,\dot{w}\rangle$. Associated with $(\mathcal{E},\Omega)$ is its Heisenberg group: the Cartesian product ${\rm H}= \R\times \mathcal{E}$ equipped with the group operation $(r,u)(\hat{r},\hat{u}) = (r+\hat{r} -\Omega(u,\hat{u}), u+\hat{u})$.

All three quantities $q$, $(q,p)$, and $C$ depend homomorphically on $\sigma$. Therefore, we have left actions of $\mathrm{S}$ on $I$, $\R$, and ${\rm C}^\infty(I,V)$, given by
\begin{equation}
  \label{eq:H_actions}
 {\rm i)}~  \sigma t = qt+p, \quad {\rm ii})~ \sigma s = q^{-1}s, \quad\mbox{and}\quad {\rm iii)}~(\sigma u)(t) = Cu(q^{-1}(t-p)),
\end{equation}
respectively. As (\ref{eq:H_actions}-i) and (\ref{eq:H_actions}-ii) are at odds when $I=\R$, we adopt only (\ref{eq:H_actions}-i) and explicitly write $q^{-1}s$ for (\ref{eq:H_actions}-ii), understanding that $q$ is the first component of $\sigma$. Note that as a particular case of (\ref{eq:H_actions}-iii), ${\rm S}$ acts on $V$ via $\sigma v = Cv$, and these actions are compatible in the sense that $(\sigma u)(t) = \sigma(u(\sigma^{-1}t))$. Finally, as a routine computation shows, (\ref{eq:H_actions}-iii) also restricts to an action of ${\rm S}$ on $\mathcal{E}$ for which $\sigma^*\Omega = q^{-1}\Omega$ for every $\sigma\in {\rm S}$, whenever $\sigma$ is regarded as a linear operator $\sigma\colon \mathcal{E}\to \mathcal{E}$.

The description of ${\rm Iso}(\widehat{M},\hg)$ given below as a semidirect product was already known to Schliebner, at least in Lorentzian signature, cf. the preprint \cite{S_arxiv}. In general indefinite metric signature, it can be found in \cite[Theorem 6.1]{DT_agag}:

\begin{theorem}\label{lem:full_iso_model}
  The isometry group of a standard ECS plane wave $(\widehat{M},\hg)$ is isomorphic to the semidirect product \hbox{$\mathrm{S}\ltimes_\rho \mathrm{H}$}, where the structure morphism \linebreak[4]\hbox{$\rho\colon \mathrm{S}\to {\rm Aut}(\mathrm{H})$} is given by $\rho(\sigma)(r,u) = (q^{-1} r, \sigma u)$. More precisely, $\varPhi \in {\rm Iso}(\widehat{M},\hg)$ corresponds to the triple $(\sigma, r, u)$ characterized by  \begin{equation} \label{eq:iso_action}
    \varPhi(t,s,v) = \left(\sigma t,\; -\langle \dot{u}(\sigma t), 2\sigma v+u(\sigma t)\rangle+q^{-1} s+r,\;\sigma v+ u(\sigma t)\right)
  \end{equation}for all $(t,s,v)\in \widehat{M}$, while the group operation in ${\rm S}\ltimes_\rho \mathrm{H}$ reads
  \begin{equation}
    (\sigma,r,u)(\hat{\sigma},\hat{r},\hat{u}) = \big(\sigma\hat{\sigma},r+q^{-1}\hat{r} - \Omega(u, \sigma\hat{u}), u+\sigma\hat{u}\big),
  \end{equation}for all $(\sigma,r,u),(\hat{\sigma},\hat{r},\hat{u}) \in {\rm S}\ltimes_\rho \mathrm{H}$.
\end{theorem}

\begin{remark}\label{rem:identify_H}
  Whenever convenient, we identify the Heisenberg group $\mathrm{H}$ with the normal subgroup $\{(1,0,{\rm Id}_V)\}\times\mathrm{H}$ of ${\rm Iso}(\widehat{M},\hg)$, which also equals the kernel of the homomorphism ${\rm Iso}(\widehat{M},\hg)\ni(\sigma,r,u)\mapsto \sigma\in\mathrm{S}$. By normality, ${\rm H}$ is invariant under all conjugation mappings ${\rm C}_\varPhi\colon {\rm Iso}(\widehat{M},\hg) \to {\rm Iso}(\widehat{M},\hg)$. If $\varPhi = (\sigma,b,w)$, the restriction ${\rm C}_\varPhi\colon {\rm H}\to {\rm H}$ is explicitly given by ${\rm C}_\varPhi(r,u) = \big(q^{-1}r-2\Omega(w,\sigma u), \sigma u\big)$.
\end{remark}

A proof of Theorem \ref{lem:full_iso_model} is given in \cite[Theorem 4.1]{DT_pm}, where we explain exactly how the proof of \cite[Theorem 2]{D_tensor} can be carried over to our current setting.

In order to explain geometric meaning of $(\mathcal{E},\Omega)$ in an abstract rank-one ECS manifold, it will be convenient to use an explicit description of $\mathfrak{iso}(\widehat{M},\hg)$. Note that the Lie algebra $\mathfrak{s}$ of ${\rm S}$ consists of all triples $(a,b,P) \in \mathfrak{aff}(\R)\times \mathfrak{so}(V,\langle\cdot,\cdot\rangle)$ having $[P,A]=2aA$ and $2af(t)+(at+b)\dot{f}(t) = 0$ for all $t\in I$.

\begin{corollary}
The Lie algebra $\mathfrak{iso}(\widehat{M},\hg)$ of Killing vector fields on a standard ECS plane wave is isomorphic to the semidirect product $\mathfrak{s} \ltimes_{\rho_*} \mathfrak{h}$, with $\rho$ as in Theorem \ref{lem:full_iso_model}. Namely, $X\in \mathfrak{iso}(\widehat{M},\hg)$ corresponds to the triple $((a,b,P),\ell,w)$ characterized by
  \begin{equation}\label{eqn:Killing_model}
    X_{(t,s,v)} = (at+b)\partial_t|_{(t,s,v)} + (-\langle\dot{w}(t),2v\rangle-as+\ell)\partial_s|_{(t,s,v)} + (Pv+w(t)),
  \end{equation}for every $(t,s,v) \in \widehat{M}$. Furthermore, the Lie bracket of two Killing vector fields $X \sim ((a,b,P), \ell,w)$ and $\hat{X} \sim ((\hat{a},\hat{b},\hat{P}), \hat{\ell},\hat{w})$ is given by
  \begin{equation}\label{eqn:Lie_bracket_Killing_model}
    [X,\hat{X}] \sim \big((0,\hat{a}b-a\hat{b}, [\hat{P},P]),\,2\Omega(w,\hat{w}) - \hat{a}\ell+a\hat{\ell}\,,\,u\big),
  \end{equation}
where $u\in\mathcal{E}$ is defined by $u(t) = (at+b)\hat{w}^{\boldsymbol\cdot}(t) - (\hat{a}t+\hat{b})\dot{w}(t)+\hat{P}w(t)-P\hat{w}(t)$.
\end{corollary}

Assume now that $(\widehat{M},\hg)$ is of rank one. The Killing vector fields tangent to the leaves of $\mathcal{D}^\perp$ are obtained by setting $a=b=0$ in \eqref{eqn:Killing_model}, and $P$ is the action of its local flow on the space of parallel sections of $\mathcal{D}^\perp/\mathcal{D}$ (so that such action is trivial precisely when $P=0$). Taking equivalence classes of these Killing vector fields modulo $\mathcal{D}$, in turn, amounts to setting $\ell=0$. The resulting quotient space is naturally identified with $\mathcal{E}$, and the Lie bracket between $X\sim ((0,0,0),0,w)$ and $\hat{X} \sim ((0,0,0),0,\hat{w})$ is readily seen to be given by $[X,\hat{X}] = 2\Omega(w,\hat{w})\partial_s$, cf. \eqref{eqn:Lie_bracket_Killing_model}. The symplectic vector space $(\mathcal{E},\Omega)$ makes sense for an abstract rank-one ECS manifold, once the objects $A$ and $f$ in \eqref{eqn:A} and \eqref{eqn:Ric_f} are in place. As a result, the symplectic form $\Omega$ is in fact a $\mathcal{D}$-valued Lie bracket, which becomes $\R$-valued once a null parallel vector field spanning $\mathcal{D}$ has been chosen.

\section{Compact quotients of standard ECS plane waves}\label{sec:cpct_ex}

\subsection{The groups ${\rm G}(\sigma)$}

As mentioned in the Introduction, the compact rank-one ECS examples presented in \cite{DR_agag} were built as isometric quotients of standard ECS plane waves, which all turned out to be of rank one. Finding subgroups $\Gamma$ of ${\rm Iso}(\widehat{M},\hg) \cong {\rm S}\ltimes {\rm H}$ (see Theorem \ref{lem:full_iso_model}) for which the quotient $\widehat{M}/\Gamma$ is smooth and compact can be as complicated as the factor ${\rm S}$ is. Thus, we restrict our search for $\Gamma$ to a certain family $\{{\rm G}(\sigma)\}_{\sigma\in {\rm S}}$ of subgroups of ${\rm Iso}(\widehat{M},\hg)$. Namely, for each $\sigma \in {\rm S}$, we let
\begin{equation}
  {\rm G}(\sigma) = \{ (\sigma^k,r,u) \mid k \in \Z \mbox{ and }(r,u)\in {\rm H}\}\cong \Z\ltimes {\rm H}.
\end{equation}
Abbreviating each $(\sigma^k,r,u)$ simply to $(k,r,u)$, the action of ${\rm G}(\sigma)$ on $\widehat{M}$ reads
\begin{equation}\label{eqn:act_Gsigma}
    (k,r,u)\cdot(t,s,v) = (\sigma^kt, -\langle \dot{u}(\sigma^kt), 2\sigma^kv + u(\sigma^kt)\rangle + q^{-k}s+r, \sigma^kv+u(\sigma^kt)), 
  \end{equation}while the operation in ${\rm G}(\sigma)$ can be written as
  \begin{equation}\label{eqn:OP-Gsigma}
    (k,r,u)(\hat{k},\hat{r},\hat{u}) = (k+\hat{k}, r+q^{-k}\hat{r} - \Omega(u,\sigma^k\hat{u}), u+\sigma^k\hat{u}),
  \end{equation}for all $(k,r,u),(\hat{k},\hat{r},\hat{u}) \in {\rm G}(\sigma)$ and $(t,s,v) \in \widehat{M}$. Here, we use the notation introduced in \eqref{eq:H_actions} and the lines following it. Lastly, inspired by the discussion in Section \ref{subsec:td-dich}, we will say that
  \begin{equation}\label{eqn:td-dich}
  \parbox{.76\textwidth}{$\sigma$ is \emph{translational} if it has the form $\sigma=(1,p,C)$ and $I=\R$, and \emph{dilational} if it has the form $\sigma=(q,0,C)$ and $I=(0,\infty)$,}  
  \end{equation}
the interval $I$ being the one used to define $(\widehat{M},\hg)$ in \eqref{eqn:model}.

  \subsection{First-order subspaces, and the existence criterion}

In both \cite[Section 7]{DT_mofm} and \cite[Section 5]{DT_advg}, for the sake of self-containedness, we have presented \emph{ad hoc} proofs that once a suitable subgroup $\Gamma$ of ${\rm Iso}(\widehat{M},\hg)$ (in fact, lying in some ${\rm G}(\sigma)$) has been found, the resulting quotient $\widehat{M}/\Gamma$ is indeed smooth and compact. Adapting the results from \cite[Sections 4--7]{DR_agag}, we establish in Theorem \ref{teo:criterion_compact} below a more general criterion for the existence of subgroups $\Gamma$ of a given ${\rm G}(\sigma)$ producing smooth compact quotients $\widehat{M}/\Gamma$.

Consider again the symplectic vector space $(\mathcal{E},\Omega)$ introduced in \eqref{eqn:symp}, associated with a standard ECS plane wave $(\widehat{M},\hg)$. A vector subspace $\mathcal{L}$ of $\mathcal{E}$ is called a \emph{first-order subspace} if, for every $t\in I$, the restriction $\delta_t|_{\mathcal{L}}\colon\mathcal{L} \to V$ is an isomorphism, where $\delta_t\colon \mathcal{E}\to V$ denotes the evaluation at $t$. The name `first-order' is explained by the next result, which is a direct generalization of \cite[Lemma 4.1]{DR_agag}:

  \begin{lemma}\label{lem:first-order-subspaces}
  First-order subspaces $\mathcal{L}$ of $(\mathcal{E},\Omega)$ are in bijective correspondence with smooth curves $B\colon I\to {\rm End}(V)$ such that $\dot{B}+B^2=f+A$, where $f$ stands for the function $t\mapsto f(t)\,{\rm Id}_V$. The correspondence assigns to $B$ the space $\mathcal{L}$ of all solutions $u\colon I\to V$ to the differential equation $\dot{u}(t) = B(t)u(t)$, and
  \begin{enumerate}[\normalfont(i)]
  \item $\mathcal{L}$ is a Lagrangian subspace of $(\mathcal{E},\Omega)$ if and only if each $B(t)$ is self-adjoint.
  \end{enumerate}
  In addition, if $\sigma =(q,p,C)$ is chosen as in \eqref{eqn:td-dich}, then
  \begin{enumerate}[\normalfont(i)]\setcounter{enumi}{1}
  \item $\mathcal{L}$ is $\sigma$-invariant if and only if $B(\sigma t) = q^{-1}CB(t)C^{-1}$ for each $t\in I$;
  \item whenever $\mathcal{L}$ is $\sigma$-invariant, the determinant of $\sigma|_{\mathcal{L}}\colon \mathcal{L}\to\mathcal{L}$ is given by
    \begin{equation}\label{eq:detT}
      \det \big(\sigma|_{\mathcal{L}}\big)  =(\det C) \exp\left(-\int_\varepsilon^{\sigma\varepsilon} {\rm tr}(B(t))\,{\rm d}t\right),
    \end{equation}where $\varepsilon$ is $0$ or $1$ according to whether $\sigma$ is translational or dilational.
  \end{enumerate}
\end{lemma}

Now, given $\sigma\in {\rm S}$, we let $\Pi\colon {\rm G}(\sigma) \to \Z$ and $\delta\colon \ker \Pi\to \mathcal{E}$ be the natural projections, given by $\Pi(k,r,u)=k$ and $\delta(0,r,u) = u$. In addition, to each subgroup $\Gamma$ of ${\rm G}(\sigma)$ we associate the following objects: \vspace{-1em}
\begin{equation}\label{eqn:objects_Gamma}
\vspace{-1em}  \parbox{.5\textwidth}{
    \begin{enumerate}[(i)]
    \item the intersection $\Sigma = \Gamma \cap \ker \Pi$.
    \item the image $\Lambda = \delta(\Sigma)$.
    \item the subspace $\mathcal{L}$ of $\mathcal{E}$ spanned by $\Lambda$.
    \end{enumerate}}
\end{equation}
Note that $\Sigma$ may be seen as a subset of ${\rm H} = \R\times \mathcal{E}$ in view of Remark \ref{rem:identify_H}, and we in fact have $\Sigma\subseteq \R\times \Lambda$.

 The next result generalizes \cite[Theorem 5.1]{DR_agag}, originally stated only for the case where $\sigma \in {\rm S}$ is translational and with $C= {\rm Id}_V$. Below, $\mathcal{P}$ is the distribution introduced in \eqref{eqn:distr_P}, and the conjugation mapping ${\rm C}_\gamma$ is as in Remark \ref{rem:identify_H}. Note that in the rank-two case, the distribution $\mathcal{P}$ also survives in the resulting compact quotients.

\begin{lemma}\label{lem:technical}
  For a standard ECS plane wave $(\widehat{M},\hg)$ and $\sigma \in {\rm S}$ chosen as in \eqref{eqn:td-dich}, let there be a subgroup $\Gamma$ of ${\rm G}(\sigma)$ acting freely and properly discontinuously on $\widehat{M}$ with a compact quotient $M = \widehat{M}/\Gamma$. Then:
  \begin{enumerate}[\normalfont(i)]
  \item There is $\theta \in [0,\infty)$ such that $\Sigma \cap (\R\times \{0\}) = \Z\theta\times \{0\}$ and $2\Omega(u,\hat{u}) \in \Z\theta$ for all $u,\hat{u}\in \Lambda$.
  \item For every $k\in \Pi(\Gamma)$, we have $\sigma^k(\Lambda)=\Lambda$ and $\sigma^k(\mathcal{L})=\mathcal{L}$. In addition, $\sigma^k|_{\mathcal{L}} = {\rm Id}_{\mathcal{L}}$ if $k$ arises from a central element of $\Gamma$, and thus $\Gamma$ is not virtually Abelian unless $\sigma^k|_{\mathcal{L}}= {\rm Id}_{\mathcal{L}}$ for some positive $k\in \Pi(\Gamma)$.
  \item There is a locally trivial fibration $M\to \esf^1$ whose fibers are the leaves of $\mathcal{P}^\perp$, all diffeomorphic to tori when $\theta = 0$ in \emph{(i)}.
  \item $\mathcal{L}$ is a $\sigma$-invariant first-order subspace of $\mathcal{E}$ and, for $\theta$ in \emph{(i)}, one of the following situations must occur:
    \begin{enumerate}[\normalfont(a)]
    \item $\theta = 0$, $\mathcal{L}$ is Lagrangian, and $\Sigma$ is a lattice in $\R\times \mathcal{L}$ projecting isomorphically onto $\Lambda$ under the projection $\R\times\mathcal{L} \to \mathcal{L}$.
    \item $\theta > 0$, while $\Lambda$ and each $\Lambda_t = \delta_t(\Sigma)$ are lattices in $\mathcal{L}$ and $V$, respectively.
    \end{enumerate}
  \item If $\sigma$ is dilational, then option \emph{(a)} in \emph{(iv)} is the one that holds and $\Gamma$ is not virtually Abelian.
  \end{enumerate}
\end{lemma}

\begin{proof}
  The proof of \cite[Theorem 5.1]{DR_agag} can be applied here, \emph{verbatim}, to obtain a major part of the claimed conclusions. Namely, (i), (ii) and (iv) here correspond to (c)-(d), (e)-(f), and (g)-(h) in \cite{DR_agag}, respectively.

  Consider (iii). It is clear that
  \begin{equation}\label{eqn:Sigma_Mt}
    \parbox{.6\textwidth}{$\Sigma$ acts freely and properly discontinuously on each hypersurface $\widehat{M}_t = \{t\}\times \R\times V$, with $t\in I$,}
  \end{equation}
  and that the projection $\widehat{M}\ni (t,s,v)\mapsto t\in I$ becomes a $\Gamma$-equivariant map (the action of $\Gamma$ on $I$ is induced by the action of ${\rm S}$, cf. (\ref{eq:H_actions}-i)). Such projection then descends to a surjective submersion $M =\widehat{M}/\Gamma \to I/\Gamma \cong \mathbb{S}^1$, whose (automatically compact) fibers are precisely the closed images of the natural embeddings $\widehat{M}_t/\Sigma \hookrightarrow \widehat{M}/\Gamma$. Here,
  \begin{equation}\label{eqn:quot_I-Gam}
    \parbox{.77\textwidth}{the quotient $\,I/\Gamma\,$ equals either $\,\R/p\Z\,$ or $\,(0,\infty)/q^\Z\,$ according to whether $\,\sigma\,$ in \eqref{eqn:td-dich} is translational or dilational, respectively, but both quotients are naturally isomorphic to the circle $\,\esf^1$.}
  \end{equation}
  By Ehresmann's fibration theorem, $M \to \mathbb{S}^1$ is a fiber bundle projection. It remains to show that, when $\theta=0$, each quotient $\widehat{M}_t/\Sigma$ is diffeomorphic to a torus. The surjective homomorphism $\delta|_{\Sigma}\colon \Sigma \to \Lambda$, having trivial kernel by (i), must be an isomorphism, allowing us to regard $\Sigma$ (now Abelian) as a discrete subgroup of $\R\times\mathcal{L}$. Next, as $\mathcal{L}$ is Lagrangian by (i), the induced Heisenberg operation in $\R\times\mathcal{L}$ coincides with its obvious vector space addition. The mapping
  \begin{equation}\label{eqn:Sigma-eq-diffeo}
    \R\times\mathcal{L} \ni (r,u) \mapsto \big(t, r-\langle\dot{u}(t),u(t)\rangle, u(t)\big) \in \widehat{M}_t
  \end{equation}is a $\Sigma$-equivariant diffeomorphism, making $(\R\times\mathcal{L})/\Sigma$ diffeomorphic to $\widehat{M}_t/\Sigma$, and hence compact. Thus $\Sigma$ is a lattice in $\R\times\mathcal{L}$ and $(\R\times\mathcal{L})/\Sigma$ is a torus.
  
  We now address (v). If $\sigma$ is dilational, $\Gamma$ must necessarily contain an element $\gamma = (k,r,u)$ with $k\neq 0$, so that $\gamma_\ast \partial_s = q^{-k}\partial_s$ with $q\neq 1$, while the complete flow $\phi$ of the Killing vector field $\partial_s$ in $\widehat{M}$ is given in ${\rm G}(\sigma)$  by $\phi(\tau,\cdot) = (0,\tau,0)$. If it were $\theta>0$, we would have $\phi(\tau,\cdot) \in \Gamma$ for some $\tau\neq 0$ by (i), leading to a contradiction \cite[Lemma 2.2]{DT_pm}: for $a=q^{-k}$ we have $|a|\neq 1$, while $\gamma^m\circ \phi(\tau,\cdot) = \phi(a^m\tau,\cdot)\circ \gamma^m$ leads to $\phi(a^m\tau,\cdot)\in \Gamma$, for every $m\in \Z$. If $\eta = {\rm sgn}(1-|a|)$ and $x\in \widehat{M}$ is arbitrary, $\{\phi(a^m\tau,x)\}_{m\geq 1}$ is a sequence in $\widehat{M}$ with mutually distinct terms for $m$ large, converging to $x$, thus precluding proper discontinuity of the action of $\Gamma$ on $\widehat{M}$. This proves that $\theta=0$. Finally, if $\Gamma$ had an Abelian subgroup $\Gamma'$ with $[\Gamma:\Gamma']<\infty$, then $\Gamma'\not\subseteq \Sigma$ as $[\Gamma\colon \Sigma]=\infty$, so that $\Pi(\Gamma')$ is not trivial, while $[\Gamma'\cap \Sigma:\Sigma] < \infty$ implies that $\Gamma'\cap \Sigma$ spans $\R\times\mathcal{L}$. A second contradiction follows: let $\gamma \in \Gamma'\smallsetminus \Sigma$ and note that while ${\rm C}_\gamma$ is trivial as $\Gamma'$ is Abelian, its first component equals $q^{-1}\neq 1$, by \eqref{eqn:dich-q} and the last line of Remark \ref{rem:identify_H}.
\end{proof}

We are ready for the main result of this section. 

\begin{theorem}\label{teo:criterion_compact}
  For a standard ECS plane wave $(\widehat{M},\hg)$ of either rank, and an isometry $\gamma = (\sigma,b,w)$ with $\sigma\in {\rm S}$ chosen as in \eqref{eqn:td-dich}, the following conditions are equivalent:
  \begin{enumerate}[\normalfont(a)]
  \item There is a discrete subgroup $\Gamma \leq {\rm G}(\sigma)$ acting freely and properly discontinuously on $\widehat{M}$ with a compact quotient $M = \widehat{M}/\Gamma$.
  \item There is a $\sigma$-invariant first-order subspace $\mathcal{L}$ of $(\mathcal{E},\Omega)$, a lattice $\Sigma \subseteq \R\times \mathcal{L}$ with ${\rm C}_\gamma(\Sigma)=\Sigma$, and $\theta \in[0,\infty)$ such that $\Sigma\cap (\R\times \{0\}) = \Z\theta\times \{0\}$ and $\Omega(u,\hat{u})\in \Z\theta$ for all $u,\hat{u} \in \Lambda$, where $\Lambda$ is the image of $\Sigma$ under the projection $\R\times\mathcal{L} \to \mathcal{L}$.
  \end{enumerate}
  If \emph{(b)} holds, $\Gamma$ in \emph{(a)} can be taken to be the group generated by $\gamma$ and $\Sigma$ and there is a locally trivial fibration $M\to \esf^1$ whose fibers, all diffeomorphic to a torus or to a $2$-step nilmanifold according to whether $\mathcal{L}$ is Lagrangian or not, are the leaves of $\mathcal{P}^\perp$. Finally, $M$ equipped with its natural quotient metric is translational and complete, or dilational and incomplete, according to whether $\sigma \in {\rm S}$ itself is translational or dilational.
\end{theorem}

\noindent \emph{Proof that {\rm (a)} implies {\rm (b):}} Assume that there is a subgroup $\Gamma$ of ${\rm G}(\sigma)$ as in (a). If $\sigma$ is dilational, then items (i), (iii), (iv-a) and (v) in Lemma \ref{lem:technical} amount to (b) with $\theta=0$, and there is nothing else to prove. If $\sigma$ is translational, we claim that replacing the compact quotient $M$ with a suitable further quotient $M' \cong M/\Z_2$, we may assume that
\begin{equation}\label{simplified_assumptions}
  \parbox{.85\textwidth}{(i) $\Omega(u,\hat{u}) \in \Z\theta$ for all $u,\hat{u} \in \Lambda$,\quad and\quad (ii) $\Sigma$ is a lattice in $\R\times\mathcal{L}$,}
\end{equation}
for $\theta$ and $\Sigma$ as in Lemma \ref{lem:technical}(i), and (\ref{eqn:objects_Gamma}-i), respectively, cf. \cite[Section 6]{DR_agag}. This completes the proof, but we justify \eqref{simplified_assumptions} below for the reader's convenience.

First, assuming that $\theta > 0$, note that the element $\zeta = (0,\theta/2,0) \in {\rm G}(\sigma)$ is central, by \eqref{eqn:OP-Gsigma} with $q=1$, and hence induces an isometry $\zeta_0$ of $M$ (equipped with the quotient metric) such that $\zeta_0 \circ \pi = \pi\circ \zeta$, with $\pi\colon \widehat{M}\to M$ being the quotient projection. As $\zeta^2\in \Gamma$, we have that \hbox{$\zeta_0^2(\pi(x)) = \pi(\zeta^2(x)) = \pi(x)$} for every $x\in \widehat{M}$, allowing us to consider the (automatically proper) isometric action of $\Z_2 = \{{\rm Id}_M,\zeta_0\}$ on $M$, which we now claim is also free. Indeed, if it were to be $\zeta_0(\pi(x)) = \pi(x)$ for some $x=(t,s,v) \in \widehat{M}$, we would have that
\begin{equation}\label{Z2free}
  \parbox{.75\textwidth}{$(0,\theta/2,0)(t,s,v) = (k,r,u)(t,s,v)$,\quad for some $(k,r,u) \in \Gamma$.}
\end{equation}
By \eqref{eqn:act_Gsigma}, comparing the $I$-components in \eqref{Z2free} immediately yields $k=0$, so that $u\in \Lambda\subseteq \mathcal{L}$, while comparing the $V$-components leads to $u(t) = 0$. Thus $u=0$, as $\mathcal{L}$ is a first-order subspace by Lemma \ref{lem:technical}(iv). However, by Lemma \ref{lem:technical}(i), $(\theta/2,0) = (r,0) \in \Z\theta\times \{0\}$ leads to $\theta=0$, which is a contradiction. Hence the quotient $M' = M/\Z_2$ is indeed smooth and compact. Writing $M' = \widehat{M}/\Gamma'$, for the subgroup $\Gamma'$ of ${\rm G}(\sigma)$ generated by $\Gamma$ and $\zeta$, we observe that $\theta' = \theta/2$ and $\Sigma'=\Sigma/2$ correspond to $\Gamma'$ as $\theta$ and $\Sigma$ corresponded to $\Gamma$, so that (\ref{simplified_assumptions}-i) follows from Lemma \ref{lem:technical}(i).

  It remains to establish (\ref{simplified_assumptions}-ii), still under the assumptions that $\sigma$ is translational and that Lemma \ref{lem:technical}\hbox{(iv-b)} holds. Fixing a basis $\{u_j\}_{j=3}^n$ of $\mathcal{L}$ which generates $\Lambda$, as well as $r_3,\ldots, r_n\in \R$ such that $(r_j,u_j) \in \Sigma$ for all $j=3,\ldots, n$, we now claim that
  \begin{equation}\label{eqn:Sigma-simplified}
    \Sigma = \Z(\theta,0) \oplus \Z(r_3,u_3) \oplus \cdots \oplus \Z(r_n,u_n).
  \end{equation}
  Denoting the right side of \eqref{eqn:Sigma-simplified} by $\Sigma''$, we first note that $\delta|_\Sigma\colon \Sigma\to \Lambda$ and $\delta|_{\Sigma''}\colon \Sigma'' \to \Lambda$ are both surjective homomorphisms, and so
  \begin{equation}\label{form-Sigma}
    \parbox{.7\textwidth}{every element of $\Sigma$ can be written in the specific form $\,(\theta,0)^\ell (r_3,u_3)^{k_3}\cdots (r_n,u_n)^{k_n}\,$, for some $\ell, k_3,\ldots, k_n\in \Z$.}
  \end{equation}Indeed, as $\delta((\theta,0)^\ell (r_3,u_3)^{k_3}\cdots (r_n,u_n)^{k_n}) = \sum_{j=3}^n k_ju_j$, one may reverse-en\-gi\-ne\-er \eqref{form-Sigma} by starting with $(r,u) \in \Sigma$, writing $u = \sum_{j=3}^n k_ju_j$ for suitable $k_3,\ldots, k_n\in \Z$, setting $(\widetilde{r},u) = (r_3,u_3)^{k_3}\cdots (r_n,u_n)^{k_n}$, and finally noting that $\delta(r,u) = \delta(\widetilde{r},u)$. (As $(\widetilde{r},u) \in \Sigma$ and $\ker(\delta|_{\Sigma}) = \Z\theta\times \{0\}$ by Lemma \ref{lem:technical}(i), it follows that $(r,u) = (\theta,0)^\ell(\widetilde{r},u)$ for some $\ell\in \Z$.)
  
  With \eqref{form-Sigma} in place, observe that the mapping \begin{equation}\label{eqn:double-homo}\R\times \Lambda \ni (r,u) \mapsto (r+\Z\theta, u) \in \R/\Z\theta \times \Lambda\end{equation} is a group homomorphism onto the direct product of $\R/\Z\theta$ and $\Lambda$, with kernel $\Z\theta\times\{0\}$, \emph{regardless of the group structure considered in $\R\times \Lambda$}: the additive structure, or the Heisenberg one. Then, given $(r,u) \in \Sigma$ written as in \eqref{form-Sigma}, consider the linear combination $(r'',u) = \ell(\theta,0) + \sum_{j=3}^n k_j(r_j,u_j)\in \Sigma''$, and observe that \eqref{eqn:double-homo} maps both $(r,u)$ --- by \eqref{eqn:OP-Gsigma} and (\ref{simplified_assumptions}-i) --- and $(r'',u)$ to the same image $(\sum_{j=3}^n k_jr_j + \Z\theta,u)$, allowing us to write $(r,u) = \ell'(\theta,0) + (r'',u)$ for some $\ell'\in \Z$ (so that $\Sigma\subseteq \Sigma''$), or $(r'',u) = (\theta,0)^{\ell''}(r,u)$ for some $\ell''\in \Z$ (so that $\Sigma''\subseteq \Sigma$). This establishes \eqref{eqn:Sigma-simplified}, and thus (\ref{simplified_assumptions}-ii), as required.

\medskip

\noindent \noindent \emph{Proof that {\rm (b)} implies {\rm (a):}} Assume that the objects $\mathcal{L}$, $\Sigma$, and $\theta$ are given as in (b). Our goal is to construct a smooth compact quotient of $\widehat{M}$. We start by showing that
\begin{equation}\label{compact_fiber}
  \parbox{.88\textwidth}{the quotient space $N = (\R\times\mathcal{L})/\Sigma$, where the lattice $\Sigma$ acts on the product \hbox{$\R\times\mathcal{L}$} by \emph{Heisenberg} left-translations, is a compact manifold.}
\end{equation}Indeed, if $\theta =0$ (or, more generally, if $\mathcal{L}$ is Lagrangian), the induced Heisenberg operation in $\R\times\mathcal{L}$ agrees with its standard vector space addition, and so \eqref{compact_fiber} follows from $\Sigma$ being a lattice in $\R\times\mathcal{L}$ (in which case $N$ is a torus). If $\theta > 0$ instead, the product $[0,\theta]\times Q$ is a compact fundamental domain for the \emph{Heisenberg} action of $\Sigma$ on $\R\times\mathcal{L}$ whenever $Q\subseteq\mathcal{L}$ is the image under the projection $\R\times \mathcal{L}\to \mathcal{L}$ of a compact fundamental domain for the \emph{additive} action of $\Sigma$ on $\R\times\mathcal{L}$: for any $(r,u) \in \R\times\mathcal{L}$ there is $(r_0,u_0) \in \Sigma$ such that $u+u_0\in Q$, and choosing $k\in \Z$ such that $k\theta+r_0+r-\Omega(u_0,u) \in [0,\theta]$, it follows from $(\theta,0)\in\Sigma$ that $(\theta,0)^k(r_0,u_0)(r,u) \in [0,\theta]\times Q$, as required.

Now, with $\varepsilon$ being $0$ or $1$ according to whether $\sigma$ in \eqref{eqn:td-dich} is translational or dilational, we let $\widetilde{w} \in \mathcal{L}$ be the unique element with $\widetilde{w}(\sigma \varepsilon) = w(\sigma\varepsilon)$ and $\widetilde{b} \in \R$ be given by $\widetilde{b} = b -\langle \dot{w}(\sigma\varepsilon)-B(\sigma\varepsilon)w(\sigma\varepsilon),w(\sigma\varepsilon)\rangle$, where $B$ corresponds to $\mathcal{L}$ as in Lemma \ref{lem:first-order-subspaces}. Note that $\widetilde{b} = b$ whenever $w=\widetilde{w}\in\mathcal{L}$. We claim that
\begin{equation}\label{phi_eq}
  \parbox{.89\textwidth}{$\phi\colon \R\times\mathcal{L} \to \R\times\mathcal{L}$ given by $\phi(r,u) = \big(q^{-1}r+\widetilde{b}-\Omega(\widetilde{w}-2w,\sigma u), \sigma u + \widetilde{w}\big)$ is a \hbox{$\Sigma$-equivariant} mapping, and induces a diffeomorphism $\Phi\colon N\to N$.}
\end{equation}
Indeed, $\sigma$-invariance of $\mathcal{L}$ says that $\phi$ is $(\R\times\mathcal{L})$-valued, while a routine computation shows that $\phi(r,u) = [{\rm C}_\gamma(r,u)](\widetilde{b},\widetilde{w})$, so that ${\rm C}_\gamma$-invariance of $\Sigma$ together with the resulting relation $\phi((r_0,u_0)(r,u)) = [{\rm C}_\gamma(r_0,u_0)] \phi(r,u)$, valid for any $(r_0,u_0)\in\Sigma$, implies $\Sigma$-equivariance of $\phi$.

We define an action of $\Z$ on $I\times N$ by setting $k\cdot (t,\Sigma(r,u)) = (\sigma^k t, \Phi^k\Sigma(r,u))$, so that the projection $I\times N \to I$ becomes $\Z$-equivariant and induces a locally trivial fibration $M \to \esf^1$, where $M = (I\times N)/\Z$. It is clear that $M$ so defined is compact. The projection $\pi\colon \widehat{M} \to M$, in turn, is then defined as the composition  \begin{equation}
    \begin{tikzcd}
\widehat{M} \arrow[r]      & I\times \R\times\mathcal{L} \arrow[r] & I\times N \arrow[r]                       & M                             \\[-1.5em]
{(t,s,v)} \arrow[r, maps to] & {(t,r,u)} \arrow[r, maps to]          & {(t,\Sigma(r,u))} \arrow[r, maps to] & {\Z\cdot(t,\Sigma(r,u)),}
\end{tikzcd}
\end{equation}where $(t,r,u)$ and $(t,s,v)$ are related via \eqref{eqn:Sigma-eq-diffeo}. Note that $\widehat{M}$ is the universal covering of $M$.

Letting $\Gamma$ be generated by $\Sigma$ and $\gamma$, it remains to be seen that the action of $\Gamma$ on $\widehat{M}$ --- given by \eqref{eqn:act_Gsigma} ---  is free and that the \hbox{$\pi$-fibers} coincide with the \hbox{$\Gamma$-orbits}. As $\gamma = (1,b,w)$ and $\Sigma$ is ${\rm C}_\gamma$-invariant, all elements of $\Gamma$ are of the form $(1,b,w)^k(0,r,u)$ with $k\in \Z$ and $(r,u) \in \Sigma$. Freeness of the $\Gamma$-action on $\widehat{M}$ is now clear: if $(1,b,w)^k(0,r,u)(t,s,v) = (t,s,v)$ for some $(t,s,v)\in \widehat{M}$, projecting onto the $I$-component yields $k=0$, so that $(0,r,u)(t,s,v)=(t,s,v)$ implies that $(r,u) = (0,0)$ by freeness of the $\Sigma$-action on each $\widehat{M}_t$ (which corresponds under \eqref{eqn:Sigma-eq-diffeo} to a manifestly-free Heisenberg action of $\Sigma$ on $\R\times\mathcal{L}$). Finally, as $\Sigma$ and $\gamma$ act on $\widehat{M}$ by deck transformations, each $\Gamma$-orbit is contained in a $\pi$-fiber; conversely, if $\pi(t,s,v) = \pi(\widehat{t},\widehat{s},\widehat{v})$, we may write $\widehat{t} = \sigma^kt$ and replace $(\widehat{t},\widehat{s},\widehat{v})$ with $(1,b,w)^{-k}(\widehat{t},\widehat{s},\widehat{v})$ if necessary to assume that $k=0$ and $t=\widehat{t}$, so that $(t,s,v)$ and $(\widehat{t},\widehat{s},\widehat{v})$ are in the same $\Sigma$-orbit (and hence in the same $\Gamma$-orbit).

Clearly, the ``type'' of $M$ (translational or dilational), equipped with its natural quotient metric, is the same as of $\sigma$ in \eqref{eqn:td-dich}; its completeness (or, incompleteness) now follows directly from \eqref{eqn:completeness_model}.

\subsection{The translational construction}

Here, we describe how Theorem \ref{teo:criterion_compact} leads to compact rank-one translational ECS manifolds \cite[Theorem A]{DT_mofm}:

\begin{theorem}\label{teo:existence_translational}
  There exist compact rank-one translational ECS manifolds of all dimensions $n\geq 5$ and all indefinite metric signatures, forming the total space of a nontrivial torus bundle over $\esf^1$ with its fibers being the leaves of $\mathcal{D}^\perp$, all geodesically complete, and none locally homogeneous. In each fixed dimension and metric signature, there is an infinite-dimensional moduli space of local-isometry types.
\end{theorem}

The proof strategy consists in \emph{simultaneously} producing a standard ECS plane wave $(\widehat{M},\hg)$ and a suitable subgroup $\Gamma$ of ${\rm G}(\sigma)$, where the element $\sigma \in {\rm S}$ is translational and of the form $\sigma = (1,p,{\rm Id}_V)$ for some $p>0$. Requiring the ${\rm O}(V,\langle\cdot,\cdot\rangle)$-component of $\sigma$ to be trivial turns out to be a mild assumption which can be generically arranged for --- see Section \ref{sec:class_res}. As in \eqref{eqn:quot_I-Gam}, we set $\esf^1 = \R/p\Z$, so that $p$-periodic functions defined on $\R$ (and valued anywhere) may be regarded as defined on $\esf^1$, and integration from $0$ to $p$ becomes integration over $\esf^1$.

Our goal is to find a $\sigma$-invariant first-order subspace $\mathcal{L}$ of $(\mathcal{E},\Omega)$, a con\-ju\-ga\-ti\-on-in\-va\-ri\-ant lattice $\Sigma \subseteq \R\times\mathcal{L}$, and some $\theta \geq 0$ satisfying certain additional conditions. By Lemma \ref{lem:first-order-subspaces}, such $\mathcal{L}$ corresponds to a smooth curve $B\colon \esf^1 \to {\rm End}(V)$ solving the differential equation $\dot{B}+B^2=f+A$. This is a crucial point: the objects $f$ and $A$ needed to define the metric $\hg$ are in fact \emph{determined} by $B$, as the scalar and traceless parts of $\dot{B}+B^2$. In addition, if each $B(t)$ is self-adjoint, $\mathcal{L}$ becomes Lagrangian and the only condition relating $\Sigma$ and $\theta$ simply reads $\Sigma = \Z\theta \times \Lambda$. Lemma \ref{lem:technical} suggests that $\Lambda$ ought to arise as the lattice preserved by $\sigma|_{\mathcal{L}} = \exp\big(-\int_{\esf^1}B\big)$ acting on $\mathcal{L}$.

In summary, it suffices to obtain a curve $B\colon \esf^1\to {\rm End}(V)$ of self-adjoint endomorphisms of $(V,\langle\cdot,\cdot\rangle)$ for which \vspace{-1em}
\begin{equation}\label{want_B}
\vspace{-1em}  \parbox{.9\textwidth}{\begin{enumerate}[(i)]
\item the trace of $\dot{B}+B^2$ is nonconstant.
\item the traceless part of $\dot{B}+B^2$ is a nonzero constant.
\item the operator $\exp\big(-\int_{\esf^1}B\big)$ maps a lattice $\Lambda\subseteq\mathcal{L}$ onto itself.
\end{enumerate}}
\end{equation}
Conditions (\ref{want_B}-i) and (\ref{want_B}-ii) ensure, respectively, that the resulting standard ECS plane wave $(\widehat{M},\hg)$ is not locally symmetric or conformally flat. Then, we \emph{define} $f$ and $A$ as the scalar and traceless parts of $\dot{B}+B^2$, the lattice $\Sigma$ as the product $\Z\theta \times \Lambda$, and finally apply Theorem \ref{teo:criterion_compact}. 

To narrow down the search for $B$ we assume that $V = \R^{n-2}$, let $\Delta^{n-2} \cong \R^{n-2}$ denote the space of diagonal matrices of order $n-2$, and consider the arithmetic condition imposed on $(\lambda_1,\ldots,\lambda_{n-2})\in \R^{n-2}$:
\begin{equation}\label{arit_prop}
  \parbox{.5\textwidth}{$\{\lambda_1,\ldots,\lambda_{n-2}\} \subseteq (0,\infty)\smallsetminus \{1\}$ is not of the form $\{\lambda\}$ or $\{\lambda,\lambda^{-1}\}$, for any $\lambda>0$.}
\end{equation}
Whenever the entries of $\varTheta\in\Delta^{n-2}$ satisfy \eqref{arit_prop}, there is an infinite-dimensional submanifold of ${\rm C}^\infty(\esf^1)$ whose elements can be realized as \hbox{$(n-2)^{-1}{\rm tr}(\dot{B}+B^2)$} for some $B\in {\rm C}^\infty(\esf^1,\Delta^{n-2})$ having $\exp(-\int_{\esf^1}B) = \varTheta$ and the traceless part of \hbox{$\dot{B}+B^2$} be a nonzero constant \cite[Theorem 6.2]{DT_mofm}. The proof of this fact combines a well-placed application of the Inverse Function Theorem in the lower-regularity Banach spaces ${\rm C}^k(\esf^1,\Delta^{n-2})$, and the existence of local smooth-pre\-ser\-ving retractions \cite[Lemma 5.1]{DT_mofm}, needed to deform the resulting ${\rm C}^k$ curves $B$ into ${\rm C}^\infty$ ones. This takes care of conditions (\ref{want_B}-i) and (\ref{want_B}-ii). As for (\ref{want_B}-iii), it suffices to note that for every $n\geq 5$ there is a matrix in ${\rm GL}_{n-2}(\Z)$ with mutually distinct eigenvalues satisfying \eqref{arit_prop}, cf. \cite[Lemma 4.1]{DT_mofm}, and hence conjugate to some $\varTheta\in \Delta^{n-2}$ --- which is then used as the starting point of the preceding argument.

\subsection{The dilational construction}\label{dil_const}

This time, we describe how Theorem \ref{teo:criterion_compact} leads to odd-dimensional compact rank-two dilational ECS manifolds.

  \begin{theorem}\label{teo:exist_dil_examples}
  There exist compact rank-two dilational ECS manifolds of all odd dimensions $n\geq 5$ and with semi-neutral metric signature, including locally homogeneous ones, forming the total space of a nontrivial torus bundle over $\esf^1$ with its fibers being the leaves of $\mathcal{P}^\perp$, all of them geodesically incomplete. In each fixed odd dimension, there is an infinite-dimensional moduli space of local-isometry types.
\end{theorem}

Here, by a pseudo-Riemannian metric with \emph{semi-neutral} metric signature, we mean one for which the difference between its positive and negative indices is at most one.

As before, the construction involves two different aspects: one analytical for $f$ and $A$, and one combinatorial for $\Gamma$. The difference is that while the arithmetic condition \eqref{arit_prop} was simple, and finding $f$ and $A$ required deforming constant solutions of an ordinary differential equation, the roles are now reversed: $f$ and $A$ arise from the very explicit formulas \eqref{AandC} and \eqref{f}, and the existence of $\Sigma$ and $\mathcal{L}$ is established via an elaborate combinatorial structure, only available in odd dimensions \cite[Section 2]{DT_advg}.

More precisely, a \emph{$\Z$-spectral system} is a quadruple $(m,k,E,J)$ where $m,k\geq 2$ are integers and, setting $\mathcal{V} = \{1,\ldots,2m\}$, $E\colon \mathcal{V} \to \Z\smallsetminus \{-1\}$ and $J\colon \mathcal{V}\to \{0,1\}$ are functions, with $E$ injective, such that: \vspace{-1em}
\begin{equation}\label{Z-spec}
\vspace{-1em}  \parbox{.9\textwidth}{
    \begin{enumerate}[(i)]
    \item $k+1=2E(1)$ (thus making $k$ odd).
    \item $E(i)+E(i')=-1$ and $J(i)+J(i')=1$ whenever $i+i'=2m+1$.
    \item $E(i)-E(i') = k$ and $J(i)+J(i')=1$ whenever $i'=i+1$ is even.
    \item the set $Y = \{-1\}\cup \{E(i)\}_{i\in J^{-1}(1)}$ is symmetric about zero.
    \end{enumerate}}
\end{equation}
The \emph{spectral selector} $S_1=J^{-1}(1)$ is simultaneously a selector for both families \vspace{-1em}
\begin{equation}\label{eqn:selector}
\vspace{-1em}  \parbox{.9\textwidth}{
    \begin{enumerate}[(i)]
    \item $\{ \{i,i'\}\in\wp_2(\mathcal{V}) \mid i+i'=2m+1\}$ and
    \item $\{\{i,i'\} \in \wp_2(\mathcal{V}) \mid i'=i+1\mbox{ is even}\}$,
    \end{enumerate}}
\end{equation}
where $\wp_2(\mathcal{V})$ stands for the collection of two-element subsets of $\mathcal{V}$. The function $E$, in turn, is to be thought of as an `exponent function', which will give rise to the correct spectrum for the dilational element $\sigma \in {\rm S}$ to be described below.

Let $q\in (0,\infty)\smallsetminus \{1\}$ and $n\geq 5$ be odd, set $m=n-2$, and fix a \hbox{$\Z$-spectral} system $(m,n,E,J)$ \cite[Theorem 2.2]{DT_advg}. In addition, let $(V,\langle\cdot,\cdot\rangle)$ be a $m$-dimensional pseudo-Euclidean vector space of semi-neutral metric signature and, for the scalars $a(1),\ldots, a(m)\in \R$ given by $a(i) = E(2i-1) + (1-n)/2$, for all $i$, we define $A$ and $C$ by 
\begin{equation}\label{AandC}
  \parbox{.75\textwidth}{$Ae_m = e_1$ and $Ae_i=0$ for $i<m$, and $Ce_i = q^{a(i)}e_i$ for all $i$.}
\end{equation}
In \eqref{AandC}, $(e_1,\ldots, e_m)$ is a basis for $V$ with the property that, for some $\varepsilon = \pm 1$, $\langle e_i,e_k\rangle = \varepsilon \delta_{ij}$ for all $i,j\in \{1,\ldots,m\}$, where $k=m+1-j$ and we set $e_0=0$. With this, we have that $A$ is self-adjoint, $C\in {\rm SO}(V,\langle\cdot,\cdot\rangle)$, and the relation $CAC^{-1}=q^2A$ holds (cf. Section \ref{sec:full_iso}).

The element $\sigma = (q,0,C)$, regarded as an endomorphism of $\mathcal{E}$ via \hbox{(\ref{eq:H_actions}-iii)}, has the spectrum $(\mu^+q^{a(1)}, \mu^-q^{a(1)},\ldots, \mu^+q^{a(m)},\mu^-q^{a(m)})$, where $(\mu^+,\mu^-)$ is the spectrum of the endomorphism $y\mapsto (t\mapsto y(t/q))$ on the space of solutions $y\colon (0,\infty)\to \R$ of the ordinary differential equation $\ddot{y} = fy$. As a consequence,
\begin{equation}\label{spec-sigma}
  \parbox{.5\textwidth}{when $\mu^\pm = q^{(-1\pm n)/2}$, the spectrum of $\sigma$ becomes precisely $(q^{E(1)},\ldots, q^{E(2m)})$,}
\end{equation}
cf. \cite[formula (4.2)]{DT_advg}. Condition \eqref{spec-sigma} can be achieved, for instance, by choosing the function \begin{equation}\label{f}f(t) = \frac{n^2-1}{4t^2}.\end{equation} We may also require the eigenfunctions associated with $\mu^+$ and $\mu^-$ to be positive. A straightforward computation --- see Theorem \ref{lem:full_iso_model} --- shows that this choice of $f$ makes the resulting standard ECS plane wave $(\widehat{M},\hg)$ homogeneous.

Let $(u_1,\ldots,u_{2m}) = (u_1^+,u_1^-,\ldots, u_m^+,u_m^-)$ be a basis of $\mathcal{E}$ consisting of eigenvectors of $\sigma$, associated with the eigenvalues $q^{E(1)},\ldots, q^{E(2m)}$, with the property that $\Omega(u_i,u_j) = 0$ whenever $i,j\in \{1,\ldots,2m\}$ have $i+j\neq 2m+1$, whose existence is ensured by \cite[Lemma 4.1]{DT_advg}. It now follows that the direct sum $\mathcal{L} = \bigoplus_{i\in S_1} \R u_i$ is a first-order \hbox{$\sigma$-in\-va\-ri\-ant} Lagrangian subspace of $(\mathcal{E},\Omega)$: \hbox{$\sigma$-invariance} is obvious, being Lagrangian is a consequence of $S_1$ being a selector for (\ref{eqn:selector}-i), and being first-order follows from the matrix representing $\sigma|_{\mathcal{L}}$ being upper triangular and with positive diagonal entries.

Finally, fix any $\gamma = (\sigma,b,w) \in {\rm G}(\sigma)$, and consider its conjugation mapping ${\rm C}_\gamma$, cf. Remark \ref{rem:identify_H}. The product $\R\times\mathcal{L}$ is ${\rm C}_\gamma$-invariant, since $\mathcal{L}$ is $\sigma$-invariant, and the spectrum of the restriction ${\rm C}_\gamma|_{\R\times\mathcal{L}}$ is given by $(q^a)_{a\in Y}$, for the set $Y$ in \hbox{(\ref{Z-spec}-iv)}. The existence of a lattice $\Sigma \subseteq \R\times\mathcal{L}$ which is mapped onto itself by ${\rm C}_\gamma$ is guaranteed whenever $q+q^{-1}\in \Z$ \cite[Remark 2.1]{DT_advg}, so starting the argument with such a value of $q$ allows us to apply Theorem \ref{teo:criterion_compact}.

The resulting compact quotients are locally homogeneous due to the choice of $f$ in \eqref{f}, but we have an infinite-dimensional freedom to deform it while keeping $\mu^+$ and $\mu^-$ the same \cite[Theorem A.1]{DT_advg}. For this reason we obtain dilational but non locally-homogeneous compact quotients as well.

\section{The bundle structure}\label{eqn:bundle_structure}

All known compact rank-one ECS examples share a specific topological feature: they are all bundles over $\esf^1$ in such a way that $\mathcal{D}^\perp$ appears as the vertical distribution. This is not accidental \cite[Theorem A]{DT_pems}:

\begin{theorem}\label{teo:bundle_pems}
  Every non locally-homogeneous compact rank-one ECS manifold $(M,\g)$ for which the orthogonal distribution $\mathcal{D}^\perp$ is transversely orientable is the total space of a locally trivial fibration over $\esf^1$ whose fibers are the leaves of $\mathcal{D}^\perp$. The conclusion remains valid in the locally homogeneous case if, in addition, one assumes that $\mathcal{D}^\perp$ has at least one compact leaf.
\end{theorem}

Observe that this generalizes \cite[Theorem B]{DR_jgp} from the Lorentzian signature, where the locally homogeneous case was ruled out (cf. Section \ref{subsec:td-dich}). The proof of Theorem \ref{teo:bundle_pems} however, uses different tools and clarifies the geometric role of $\mathcal{D}^\perp$. The precise differences between the new and old proofs in the Lorentzian case are explained in the Appendix of \cite{DT_pems}.

The central concept used in what follows is what we call the \emph{dichotomy property} \cite{D_rjm} for a codimension-one foliation $\mathcal{V}$ in a smooth manifold $M$, which has two alternatives (NC) and (AC) imposed on its compact leaves. Namely, $\mathcal{V}$ has the dichotomy property if every compact leaf $L$ of $\mathcal{V}$ has a neighborhood $U$ in $M$ such that the leaves of $\mathcal{V}$ intersecting $U\smallsetminus L$ are either:  \vspace{-1em}
\[\vspace{-1em} \parbox{.9\textwidth}{\begin{itemize}
\item[(NC)] all noncompact, or
\item[(AC)] all compact, and some neighborhood of $L$ in $M$ saturated by compact leaves of $\mathcal{V}$ may be diffeomorphically identified with the product $\R\times L$ in such a way that $\mathcal{V}$ corresponds to the foliation $\{\{s\}\times L\}_{s\in \R}$.
\end{itemize}} \]
The reason why we care about this property is that if $M$ is compact, $\mathcal{V}$ is transversely orientable and has the dichotomy property, and there is one compact leaf of $\mathcal{V}$ satisfying (AC), then there is a locally trivial fibration $M\to \esf^1$ whose fibers are the leaves of $\mathcal{V}$ \cite[Theorem 4.1]{DT_pems} --- the main argument justifying it goes back to Reeb.

In view of the above, there are two main steps to be carried out: \vspace{-1em}
\begin{equation}\label{steps}
\vspace{-1em}  \parbox{.9\textwidth}{\begin{enumerate}[(i)]
\item showing that $\mathcal{D}^\perp$ (when transversely orientable) has the dichotomy property, and
\item establishing the existence of a compact leaf of $\mathcal{D}^\perp$ satisfying (AC).
\end{enumerate}}
\end{equation}
While (\ref{steps}-i) does not rely on compactness of $M$, we could only achieve (\ref{steps}-ii) outside of the locally homogeneous case.

Let $(M,\g)$ be a compact rank-one ECS manifold. Replacing $(M,\g)$ with a suitable isometric double covering if needed, we may assume that $\mathcal{D}^\perp$ is transversely orientable. The crucial fact needed to conclude (\ref{steps}-i) is that whenever $L$ is a compact leaf of $\mathcal{D}^\perp$, some neighborhood $U$ of $L$ in $M$ can be diffeomorphically identified with a neighborhood $U'$ of the zero section $L$ in the line bundle $\mathcal{D}^*_L$ (the dual bundle of $\mathcal{D}$ restricted to $L$) in such a way that the distribution $\mathcal{D}^\perp$ on $U$ corresponds to the restriction to $U'$ of the horizontal distribution of the natural flat connection induced in $\mathcal{D}^*_L$ \cite[Theorem 10.1]{DT_pems}. Once this in place, the dichotomy property for $\mathcal{D}^\perp$ follows, with options (AC) and (NC) corresponding to whether the holonomy group of the natural flat connection in $\mathcal{D}^*_L$ is finite or infinite, respectively.

The argument for (\ref{steps}-ii) is more elaborate, and we need to consider cohomology classes of continuous $1$-forms. Namely, a continuous $1$-form is called \emph{closed} if it is locally given by the differential of a ${\rm C}^1$ function, in which case it indeed gives rise to a cohomology class (that is itself trivial if and only if a global ${\rm C}^1$ anti-derivative exists). The cohomology space $H^1_{\rm dR}(M)$ obtained in such a way can be identified with ${\rm Hom}(\pi_1(M),\R)$, as usual.

Fix the universal covering $(\widetilde{M},\tg)$ of $(M,\g)$, as well as a smooth function $t \colon \widetilde{M}\to \R$ whose parallel gradient spans $\mathcal{D}$, cf. \eqref{eqn:flat_D_Dquot} and \eqref{t}. Writing $M = \widetilde{M}/\Gamma$ for some group $\Gamma \cong \pi_1(M)$ (see Section \ref{subsec:td-dich}) and introducing the space $\mathcal{F}$ of all $\chi \in {\rm C}^0(\widetilde{M})$ for which $\chi\,{\rm d}t$ is closed and $\Gamma$-invariant, we may consider the linear operator $P\colon \mathcal{F} \to H^1_{\rm dR}(M)$ given by  $P\chi = [\chi\,{\rm d}t]$. Here, by $[\chi\,{\rm d}t]$ we mean the cohomology class of the $1$-form in $M$ whose pullback under the universal covering projection is $\chi\,{\rm d}t$. Using this operator $P$, we may show that if $(M,\g)$ is not locally homogeneous,
\begin{equation}\label{special_mu}
  \parbox{.78\textwidth}{there is a nonconstant $\mu\in {\rm C}^1(M)$ which is constant along $\mathcal{D}^\perp$,}
\end{equation}
cf. \cite[Theorem 9.1]{DT_pems}. More precisely, either $\dim \mathcal{F} < \infty$ and $(M,\g)$ is locally homogeneous, or $\dim \mathcal{F} = \infty$ and $\mu$ in \eqref{special_mu} exists --- the latter case is justified by choosing $\chi \in \ker P\smallsetminus \{0\}$ and letting $\mu\in {\rm C}^1(M)$ be an anti-derivative of the (now exact) $1$-form in $M$ induced by $\chi\,{\rm d}t$. The first case is ruled out due to set-theoretical reasons \cite[Lemma 3.3]{DT_pems} ultimately implying that $f(t) = \varepsilon(t-b)^{-2}$ for some $b\in \R$ and $\varepsilon = \pm 1$, for the function $f$ characterized by ${\rm Ric} = (2-n)f(t)\,{\rm d}t\otimes {\rm d}t$ in $\widetilde{M}$. (Such formula for $f$ implies local homogeneity of $(M,\g)$, by \cite[Theorem 7.3]{DT_pems}.)

With $\mu$ as in \eqref{special_mu} in hands, the conclusion of (\ref{steps}-ii) follows from Sard's theorem: the image of $\mu$ in $\R$ contains an open interval of regular values of $\mu$; any connected component of a level set $\mu^{-1}(c)$, with $c$ in a such open interval, is a compact leaf of $\mathcal{D}^\perp$ satisfying (AC). It is worth noting that while Sard’s theorem usually applies for a ${\rm C}^k$ function from an $n$-manifold into an $m$-manifold, where $k \geq \max\{n - m + 1, 1\}$, compactness of $\mu$ together with $\mu$ being locally a function of the real-valued function $t$ (which is without critical points) allows us to apply Sard with $n=1$ as opposed to $n\geq 4$, cf. \cite[Remark 9.2]{DT_pems}.

\section{Genericity and classification results}\label{sec:class_res}

The known classification results for compact rank-one ECS manifolds, which we briefly discuss\footnote{Most proof outlines are omitted due to space limitations.} in this final section, rely on a certain notion of \emph{genericity}, explained next.

Namely, if $(V,\langle\cdot,\cdot\rangle)$ is a pseudo-Euclidean vector space and $A\colon V\to V$ is traceless and self-adjoint, we say that $A$ is \emph{generic} if only finitely many isometries of $(V,\langle\cdot,\cdot\rangle)$ commute with $A$. For example, all diagonalizable operators with mutually distinct eigenvalues are generic. As shown in \cite[Section 4]{DT_agag}, every nonzero $A$ is generic when $\dim V = 2$, and generic endomorphisms form an open and dense subset of the space of all traceless and self-adjoint operators. Still at the linear algebra level, the algebraic structure of nilpotent generic operators is known \cite[Theorem 5.1]{DT_pm} and, in this case, for each $q\in (0,\infty)$ there are only two isometries $C,-C$ of ($V,\langle\cdot,\cdot\rangle)$ such that $CAC^{-1}=q^2A$ \cite[Corollary 5.3]{DT_pm}.

Proceeding, a rank-one ECS manifold $(M,\g)$ is \emph{generic} if the operator $A$ in \eqref{eqn:A} is generic in the sense described in the previous paragraph. When $(M,\g)$ is a standard ECS plane wave, this is the same as requiring the operator $A$ used to define it as in Section \ref{sec:models} to be generic. In particular, note that every four-dimensional rank-one ECS manifold is generic.

One of the main reasons genericity plays a central role in obtaining classifications results is given in \cite[Corollary D]{DT_agag}:

\begin{theorem}
  The universal covering of a generic compact rank-one ECS manifold is globally isometric to a standard ECS plane wave.
\end{theorem}

This result had been proved in the Lorentzian case by Schliebner, in the preprint \cite{S_arxiv}, without the genericity assumption. With the aid of an impossible combinatorial structure \cite[Section 8]{DT_pm}, similar in spirit to the $\Z$-spectral systems mentioned in Section \ref{dil_const}, we show that in the generic case, the image of (\ref{eqn:two_homs}-ii) is not infinite cyclic, leading to \cite[Theorem C]{DT_pm}:

\begin{theorem}\label{thm:pm23}
  A generic compact rank-one ECS manifold is either translational, or locally homogeneous.
\end{theorem}

In other words, a generic compact rank-one ECS manifold is dilational if and only if it is locally homogeneous. The locally homogeneous alternative itself, in turn, had to be ruled out with a different array of arguments, relying on particular features of the isometry group of a homogeneous standard ECS plane wave \cite[Theorem E]{DT_agag}:

\begin{theorem}\label{thm:agag23}
  All generic compact rank-one ECS manifolds are translational, complete, and not locally homogeneous. 
\end{theorem}

While Theorem \ref{thm:agag23} improves Theorem \ref{thm:pm23}, it does not replace it: the latter is crucially used to prove the former.

We conclude with one important consequence of Theorem \ref{thm:agag23}. If $(M,\g)$ is a compact four-dimensional rank-one ECS manifold then, being generic, it must be translational and have a standard ECS plane wave as its universal covering. Replacing $(M,\g)$ with a suitable isometric finite covering, we may assume that (\ref{eqn:two_homs}-i) becomes a homomorphism $\Gamma\ni \gamma \mapsto p\in \R$ (whose image is infinite cyclic), and that $\Gamma$ acts trivially on $(V,\langle\cdot,\cdot\rangle)$. It follows that $\Gamma\subseteq {\rm G}(\sigma)$ for some $\sigma \in {\rm S}$ of the form $\sigma = (1,p,{\rm Id}_V)$. With a direct adaptation of \cite[Lemma 8.1]{DR_agag}, we may apply Theorem \ref{teo:criterion_compact} to obtain our last result \cite[Corollary F]{DT_agag}:

\begin{corollary}
  There are no four-dimensional compact rank-one ECS manifolds.
\end{corollary}

%
%

\end{document}